\documentclass[11pt]{article}
\usepackage[margin=1in]{geometry} 
\usepackage{amssymb,amsfonts,amsmath,bbm,mathrsfs,stmaryrd,mathtools}
\usepackage{xcolor}
\usepackage{url}
\usepackage{relsize}

\usepackage[shortlabels]{enumitem}
\usepackage{tensor}

\usepackage{accents}

\usepackage{xr}
\usepackage[colorlinks,
linkcolor=black!75!red,
citecolor=blue,
pdftitle={},
pdfauthor={x},
pdfproducer={pdfLaTeX},
pdfpagemode=None,
bookmarksopen=true,
bookmarksnumbered=true,
backref=page]{hyperref}

\usepackage{tikz}
\usepackage{tikz-cd}
\usetikzlibrary{arrows,calc,decorations.pathreplacing,decorations.markings,intersections,shapes.geometric,through,fit,shapes.symbols,positioning,decorations.pathmorphing}

\usepackage{braket}
\usepackage{tensor}

\usepackage[amsmath,thmmarks,hyperref]{ntheorem}
\usepackage{cleveref}

\newcommand\nthalias[1]{\AddToHook{env/#1/begin}{\crefalias{lemma}{#1}}}

\nthalias{definition}
\nthalias{example}
\nthalias{examples}
\nthalias{remark}
\nthalias{remarks}
\nthalias{convention}
\nthalias{notation}
\nthalias{construction}
\nthalias{theoremN}
\nthalias{propositionN}
\nthalias{corollaryN}
\nthalias{lemma}
\nthalias{proposition}
\nthalias{corollary}
\nthalias{theorem}
\nthalias{conjecture}
\nthalias{question}
\nthalias{assumption}

\usepackage{comment}
\usepackage{setspace}

\creflabelformat{enumi}{#2#1#3}

\crefname{section}{Section}{Sections}
\crefformat{section}{#2Section~#1#3} 
\Crefformat{section}{#2Section~#1#3} 

\crefname{subsection}{\S}{\S\S}
\crefformat{subsection}{#2\S#1#3} 
\Crefformat{subsection}{#2\S#1#3} 

\theoremstyle{plain}

\newtheorem{lemma}{Lemma}[section]
\newtheorem{proposition}[lemma]{Proposition}
\newtheorem{corollary}[lemma]{Corollary}
\newtheorem{theorem}[lemma]{Theorem}

\theoremstyle{nonumberplain}

\theoremstyle{plain}
\theorembodyfont{\upshape}
\theoremsymbol{\ensuremath{\blacklozenge}}

\newtheorem{definition}[lemma]{Definition}
\newtheorem{example}[lemma]{Example}

\newtheorem{remark}[lemma]{Remark}
\newtheorem{remarks}[lemma]{Remarks}

\newtheorem{notation}[lemma]{Notation}

\crefname{definition}{definition}{definitions}
\crefformat{definition}{#2definition~#1#3} 
\Crefformat{definition}{#2Definition~#1#3} 

\crefname{ex}{example}{examples}
\crefformat{example}{#2example~#1#3} 
\Crefformat{example}{#2Example~#1#3} 

\crefname{exs}{example}{examples}
\crefformat{examples}{#2example~#1#3} 
\Crefformat{examples}{#2Example~#1#3} 

\crefname{remark}{remark}{remarks}
\crefformat{remark}{#2remark~#1#3} 
\Crefformat{remark}{#2Remark~#1#3} 

\crefname{remarks}{remark}{remarks}
\crefformat{remarks}{#2remark~#1#3} 
\Crefformat{remarks}{#2Remark~#1#3} 

\crefname{convention}{convention}{conventions}
\crefformat{convention}{#2convention~#1#3} 
\Crefformat{convention}{#2Convention~#1#3} 

\crefname{notation}{notation}{notations}
\crefformat{notation}{#2notation~#1#3} 
\Crefformat{notation}{#2Notation~#1#3} 

\crefname{table}{table}{tables}
\crefformat{table}{#2table~#1#3} 
\Crefformat{table}{#2Table~#1#3}

\crefname{lemma}{lemma}{lemmas}
\crefformat{lemma}{#2lemma~#1#3} 
\Crefformat{lemma}{#2Lemma~#1#3} 

\crefname{proposition}{proposition}{propositions}
\crefformat{proposition}{#2proposition~#1#3} 
\Crefformat{proposition}{#2Proposition~#1#3} 

\crefname{corollary}{corollary}{corollaries}
\crefformat{corollary}{#2corollary~#1#3} 
\Crefformat{corollary}{#2Corollary~#1#3} 

\crefname{theorem}{theorem}{theorems}
\crefformat{theorem}{#2theorem~#1#3} 
\Crefformat{theorem}{#2Theorem~#1#3} 

\crefname{enumi}{}{}
\crefformat{enumi}{#2#1#3}
\Crefformat{enumi}{#2#1#3}

\crefname{conjecture}{conjecture}{conjectures}
\crefformat{conjecture}{#2conjecture~#1#3} 
\Crefformat{conjecture}{#2Conjecture~#1#3} 

\crefname{assumption}{assumption}{Assumptions}
\crefformat{assumption}{#2assumption~#1#3} 
\Crefformat{assumption}{#2Assumption~#1#3} 

\crefname{construction}{construction}{Constructions}
\crefformat{construction}{#2construction~#1#3} 
\Crefformat{construction}{#2Construction~#1#3} 

\crefname{equation}{}{}
\crefformat{equation}{(#2#1#3)} 
\Crefformat{equation}{(#2#1#3)}

\numberwithin{equation}{section}
\renewcommand{\theequation}{\thesection-\arabic{equation}}

\theoremstyle{nonumberplain}
\theoremsymbol{\ensuremath{\blacksquare}}

\newtheorem{proof}{Proof}

\newcommand\bC{{\mathbb C}}
\newcommand\bD{{\mathbb D}}

\newcommand\bG{{\mathbb G}}

\newcommand\bN{{\mathbb N}}

\newcommand\bP{{\mathbb P}}

\newcommand\bR{{\mathbb R}}
\newcommand\bS{{\mathbb S}}
\newcommand\bT{{\mathbb T}}

\newcommand\bZ{{\mathbb Z}}

\newcommand\cE{{\mathcal E}}
\newcommand\cF{{\mathcal F}}

\newcommand\cM{{\mathcal M}}

\newcommand\cO{{\mathcal O}}

\DeclareMathOperator{\id}{id}
\DeclareMathOperator{\ix}{ix}
\DeclareMathOperator{\qt}{\mathrm{qt}}
\DeclareMathOperator{\End}{\mathrm{End}}
\DeclareMathOperator{\Hom}{\mathrm{Hom}}

\newcommand\numberthis{\addtocounter{equation}{1}\tag{\theequation}}

\newcommand{\cat}[1]{\textsc{#1}}

\newcommand{\ben}[1]{{\bf \color{red} **{#1}**}}

\def\polhk#1{\setbox0=\hbox{#1}{\ooalign{\hidewidth
      \lower1.5ex\hbox{`}\hidewidth\crcr\unhbox0}}}

\newcommand{\trivdim}[2]{\mathrm{dim}^{#2}_{\mathrm{LT}}(#1)}
\newcommand{\wtrivdim}[2]{\mathrm{dim}^{#2}_{\mathrm{WLT}}(#1)}
\newcommand{\strivdim}[2]{\mathrm{dim}^{#2}_{\mathrm{SLT}}(#1)}

\newcommand{\bes}{\begin{equation*}}
  \newcommand{\ees}{\end{equation*}}
\newcommand{\be}{\begin{equation}}
  \newcommand{\ee}{\end{equation}}
\newcommand{\stars}[1]{\operatorname{Star}_{#1}}

\newcommand{\xrightarrowdbl}[2][]{%
  \xrightarrow[#1]{#2}\mathrel{\mkern-14mu}\rightarrow
}

\newcommand{\thetasphere}[2]{C(\bS^{#1}_{#2})}

\allowdisplaybreaks

\begin{document}

\title{Continuity and equivariant dimension}
\author{Alexandru Chirvasitu and Benjamin Passer}

\date{}

\newcommand{\Addresses}{{
    \bigskip
    \footnotesize

    \textsc{Department of Mathematics, University at Buffalo, Buffalo,
      NY 14260-2900, USA}\par\nopagebreak \textit{E-mail address}:
    \texttt{achirvas@buffalo.edu}

    \medskip
    
    \textsc{Department of Mathematics, United States Naval Academy, Annapolis, MD 21402-5002, USA}\par\nopagebreak \textit{E-mail address}: \texttt{passer@usna.edu}
  }}

\maketitle

\begin{abstract}
  We study the local-triviality dimensions of actions on $C^*$-algebras, which are invariants developed for noncommutative Borsuk-Ulam theory. While finiteness of the local-triviality dimensions is known to guarantee freeness of an action, we show that free actions need not have finite weak local-triviality dimension. Moreover, the local-triviality dimensions of a continuous field may be greater than those of its individual fibers, and the dimensions may fail to vary continuously across the fibers. However, in certain circumstances upper semicontinuity of the weak local-triviality dimension is guaranteed. We examine these results and counterexamples with a focus on noncommutative tori and noncommutative spheres, both in terms of computation and theory.
\end{abstract}


\vspace{.5cm}



\section{Introduction}\label{se:intro}

The Borsuk-Ulam Theorem states that if $f: \bS^n \to \bS^n$ is a continuous odd function, then the degree of $f$ is nonzero. Equivalently, there is no continuous odd function $g: \bS^n \to \bS^m$ where $m < n$. While there are many equivalent forms of the theorem, each such form generalizes slightly differently into the noncommutative setting, meaning the study of unital $C^*$-algebras in place of compact Hausdorff spaces. The antipodal action $\vec{x} \mapsto -\vec{x}$ on the sphere $\bS^n$ induces an action of $\bZ/2$ on the $C^*$-algebra $C(\bS^n)$ of continuous complex-valued functions, which in turn provides a grading of the same $C^*$-algebra. This is the familiar decomposition of functions into \lq\lq even\rq\rq\hspace{0pt} and \lq\lq odd\rq\rq\hspace{0pt} components. Similarly, an odd function $g: \bS^n \to \bS^m$ may be split into individual coordinates: a list of odd, real-valued functions $g_0, \ldots, g_m \in C(\bS^n)$ with $g_0^2 + g_1^2 + \ldots + g_m^2 = 1$. Thus, the Borsuk-Ulam Theorem places a lower bound on the number of functions $g_i$ that must appear.

In the general noncommutative setting, key elements of the above discussion translate into conditions on freeness or equivariant dimension. Precursors to the study of equivariant dimension appeared in \cite[\S 3]{taghavi}, which motivated some later problems in noncommutative Borsuk-Ulam theory. Taghavi's questions primarily concerned counting the number of odd self-adjoint elements required to produce an invertible square-sum. It should be noted that this is not the only interpretation of noncommutative Borsuk-Ulam theory; references \cite{qdeformed, BDH, bentheta, benfree, alexbenjoin, alexbeninvariants} consider an approach focused on (non-)existence of equivariant morphisms. The full formulation of equivariant dimension, including a much more general context, appeared in \cite{hajacindex_xv2}. The local-triviality dimension of an action is defined in \cite[Definition 3.1]{hajacindex_xv2}, and it is a noncommutative analogue of the Schwarz genus. Moreover, it is related to the Rokhlin dimension of \cite[Definition~1.1]{dimrok}, as explored in \cite[\S 5.2]{hajacindex_xv2}. We recall the definitions of the local-triviality dimension in \Cref{se:secont}. 

In this manuscript, we will typically focus on actions of $\bS^1$ or its finite cyclic subgroups, for which equivalent versions of the dimensions we study are found in \cite[\S 3]{cpt_dim}. Consider, for example, a unital $C^*$-algebra $A$ on which $\bZ/2$ acts. The action decomposes $A$ into eigenspaces for eigenvalues $\pm 1$; elements in the $-1$ eigenspace are called \textit{odd} and elements in the $1$ eigenspace are called \textit{even}. The \textit{local-triviality (LT for short) dimension} of this action is 

\begin{equation*} \trivdim{A}{\bZ/2} = \text{inf}\left\{n \in \bN: \exists \text{ odd self-adjoint } a_0, \ldots, a_n \in A, \sum\limits_{i=0}^n a_i^2 = 1 \right\},\end{equation*}
where we note that the value $\infty$ is possible. With this notation, the Borsuk-Ulam Theorem is equivalent to the claim that $\trivdim{C(\bS^n)}{\bZ/2} = n$, where $\bZ/2$ acts on $C(\bS^n)$ via the antipodal action. Because $C(\bS^n)$ is commutative, there are adjustments to this dimension that make no difference here, but vastly change the computation for noncommutative $C^*$-algebras. For example, one could relax the condition on the sum-square of the $a_i$ to be merely invertible instead of $1$. Alternatively, one could strengthen the requirements by insisting that the $a_i$ are normal commuting elements of $A$.  Of particular note is that finiteness of the local-triviality dimension implies freeness \cite[Theorem 3.8]{hajacindex_xv2}, though this is not always equivalent.

Here, we will study how the local-triviality dimension and its variants behave with respect to deformations and fields of $C^*$-algebras, with a focus on noncommutative spheres and tori. We are interested in circumstances where the local-triviality dimensions are continuous, or not continuous, with respect to a deformation parameter. By construction of examples, we find that the most one can generally hope for is \textit{semicontinuity} of the dimension, as explored in \Cref{se:secont}. Further, we extend work of \cite{cpt_dim} by finding free actions whose weak local-triviality dimension is infinite. In \Cref{se:rat}, we explore the $\theta$-deformed spheres and tori in more detail, with a focus on the rational case. In particular, viewing a rational noncommutative torus as a space of sections of a finite-dimensional vector bundle yields information about the strong local-triviality dimension. These results strengthen the noncommutative Borsuk-Ulam theory of \cite{bentheta} when the domain sphere is commutative.


\section{Preliminaries}\label{se:prel}

For what follows, $\mathrm{Prim}(A)$ denotes the primitive spectrum \cite[\S 4.1.2]{ped_auto-2e} of a $C^*$-algebra $A$: the space of primitive ideals, equipped with the Jacobson (or hull-kernel) topology \cite[\S 4.1.4]{ped_auto-2e}. Further, $Z(A)$ denotes the center of $A$, and $M(A)$ denotes the multiplier algebra of $A$. Following \cite[\S 3]{rief_canc}, say, we write
\begin{equation*}
  e(x):=\exp(2\pi i x). 
\end{equation*}

Many of our examples concern noncommutative spheres, found by replacing commutativity of generators with a scaled commutativity relation. In \cite{no_sph}, Natsume and Olsen define a family of noncommutative spheres by a universal construction, generalizing the $3$-dimensional spheres of Matsumoto \cite{mats_3sph-1}. We will adjust their notation to be consistent with that used for noncommutative tori. For a real $n \times n$ antisymmetric matrix $\theta$, the unital $C^*$-algebra $\thetasphere{2n-1}{\theta}$ is defined by generators $z_1, \ldots, z_n$ and relations
\begin{equation}\label{eq:sph}
    z_j z_j^* = z_j^* z_j,\quad \quad z_kz_j = e(\theta_{jk}) z_j z_k, \quad \quad  z_1 z_1^* + \ldots + z_nz_n^* = 1.
\end{equation} 
If the final generator $z_n$ of $\thetasphere{2n-1}{\theta}$ is central, then one also defines $\thetasphere{2n-2}{\theta}$ by imposing the additional relation $z_n = z_n^*$. While it is also possible to consider the case that $z_n$ anticommutes with other generators (see \cite{benanticommuting, bencrossed}), the resulting $C^*$-algebras are not deformations of the commutative even-dimensional spheres in any obvious way, and we will not consider them here. 

The action of $\bZ/2$ that negates each generator will be called the \textit{antipodal action}. Similarly, there is a $\bZ/k$-action that scales each generator by $\zeta = e(1 / k)$. When $k$ is understood, we will denote the order $k$ isomorphism that implements this action by $R$, so
\begin{equation*}
  R: \, z_j \mapsto \zeta z_j.
\end{equation*}

The Natsume-Olsen spheres, which we will also call $\theta$-spheres for brevity, have $K_1(\thetasphere{2n-1}{\theta}) \cong \bZ$ \cite[Theorem 4.3]{no_sph}, with a specific generator that is inspired by the commutative case.  

\begin{remark}\label{re:s3k1gen}
  Recall (e.g. \cite[p.1094]{no_sph}) that for the classical 3-sphere, the group $K_1(C(\bS^3))$ is generated by the element
  \begin{equation}\label{eq:s3k1gen}
    \begin{pmatrix}
      \phantom{-}z_1 & z_2 \\
      -z_2^* & z_1^*
    \end{pmatrix}
    \in
    U_2(C(\bS^3))
  \end{equation}
  where $z_i$ are the standard generators of $C(\bS^3)$. Alternatively: identify $\bS^3$ with the Lie group $SU(2)$, whereupon \Cref{eq:s3k1gen} is nothing but the obvious embedding
  \begin{equation*}
    \left(
      \bS^3\cong SU(2)
      \lhook\joinrel\xrightarrow{\quad}
      M_2
    \right)
    \in
    C(\bS^3)\otimes M_2\cong M_2(C(\bS^3)).
  \end{equation*}
\end{remark}

More generally, \cite[\S 5]{no_sph} shows that the cyclic group $K_1(\thetasphere{2n-1}{\theta}) \cong \bZ$ admits a generator of matrix dimension $2^{n-1}$, which we will denote $Z_\theta$, such that each entry is a $*$-monomial. For each of the $\bZ/k$ rotation actions defined above, the eigenspace of each entry is such that there exist unitary matrices $A, B \in M_{2^{n-1}}(\bC)$, independent of $\theta$, such that $R(Z_\theta) = A Z_\theta B$  (where $R$ is applied entrywise) and $A^k = B^k = I$. The antipodal action ($k = 2$) is particularly simple; in this case each entry of $Z_\theta$ is odd, so $R(Z_\theta) = -Z_\theta$. See \cite[\S 3]{bentheta} for a slightly modified construction that chooses the monomial coefficients continuously.

The Borsuk-Ulam Theorem extends to the Natsume-Olsen spheres by \cite[Corollary 4.12]{bentheta}; if $\phi: \thetasphere{2n-1}{\theta} \to \thetasphere{2n-1}{\omega}$ is $\bZ/k$-equivariant, then the $K_1$ degree $m$ of $\phi$ is congruent to $1$ modulo $k$. In particular, the induced map $K_1(\thetasphere{2n-1}{\theta}) \to K_1(\thetasphere{2n-1}{\omega})$ is nontrivial. However, the local-triviality dimensions of these actions are not all identical. In particular, \cite[Proposition 3.21]{alexbenjoin} shows that the $\theta$-sphere $\thetasphere{2n-1}{\theta}$ such that each generator anticommutes has local-triviality dimension $1$. Since the local-triviality dimension is integer-valued and the commutative sphere has dimension $2n-1$, this immediately shows that the local-triviality dimension fails to be continuous in the parameter $\theta$.

The $\theta$-spheres are related to a similar family of $C^*$-algebras called {\it quantum (or noncommutative) tori} (\cite[\S 12.2]{gvf_ncg} or \cite[\S 1]{rief_case}). If $\theta$ is an antisymmetric $n \times n$ real matrix, the $C^*$-algebra $A^n_\theta$, or $C(\mathbb{T}^n_\theta)$, is defined by generators $u_1, \ldots, u_n$ and relations
\begin{equation}\label{eq:tor}
    u_j u_j^* = u_j^* u_j = 1
\end{equation}
\begin{equation*}
  u_k u_j = e(\theta_{jk}) u_j u_k.
\end{equation*}
Since the set of $2$-dimensional antisymmetric matrices is one-dimensional, we abuse notation somewhat when $n = 2$, and simply refer to $\theta = \theta_{12} \in \mathbb{R}$. The noncommutative tori and noncommutative spheres are related by polar decomposition \cite[Theorem 2.5]{no_sph}, which we recall in \Cref{se:rat}. In particular, the action $R$ discussed on $\theta$-spheres factors through a related action on the quantum tori:
\begin{equation*} R: u_j \mapsto \zeta u_j. \end{equation*}

All $\theta$-spheres and noncommutative tori are known to be Rieffel deformations in the sense of \cite[\S 10.2]{rief_def-rd}. While it is not clear if the $\theta$-spheres or noncommutative tori form a continuous field over the set of antisymmetric matrices,  \cite[Theorem 8.13]{rief_def-rd} implies that the field is separately continuous in each entry. In particular, the distinction between those two cases does not matter for $3$-spheres or $2$-tori. As we will discuss continuity of fields in more detail, we recall some basic definitions.

Recall \cite[Definition 1.5]{kk2} that for a locally compact Hausdorff space $X$, a {\it $C_0(X)$-algebra} is a $C^*$-algebra $A$ equipped with a non-degenerate morphism $C_0(X)\to Z(M(A))$. For us, $X$ will typically be compact and $A$ will be unital, so there will be little need to worry about units. We will also revert to the notation $C(X)=C_0(X)$ when $X$ is compact.

\begin{definition}\label{def:gcx}
  For a group $\bG$ and a locally compact Hausdorff space $X$, a {\it $\bG$-$C_0(X)$-algebra} is a $C_0(X)$-algebra equipped with an action of $\bG$ by automorphisms that fix the structure morphism $C_0(X)\to Z(M(A))$.
\end{definition}

$C(X)$-algebras are one way to formalize the notion of a {\it field} of $C^*$-algebras \cite[Definition 1.1 and Propositions 1.2 and 1.3]{rief_flds}, while \cite[Theorem 2.3]{zbMATH00975912} equates the notion to that of an upper semicontinuous $C^*$-algebra {\it bundle}, generalizing \cite[Definition VII.8.2]{fd_bdl-1}.

For $x\in X$, we write $A_x$ for the {\it fiber}
\begin{equation*}
  A_x:=A/J_x,\quad J_x:=\text{ideal generated by }\{f\in C(X)\ |\ f(x)=0\}
\end{equation*}
of the $C(X)$-algebra. We denote the quotient map $A \to A_x$ by $\pi_x$ and use $ \|\cdot\|_x$ to refer to the norm on $A_x$. Moreover, we use $\|\pi_x(a)\|_x$ and $\|a\|_x$ interchangeably, viewing $\|\cdot\|_x$ as a seminorm on $A$. The values of these seminorms do not need to vary continuously in $x \in X$, but if they do, then the field itself is called continuous (see the discussion after \cite[Definition 1.2]{blnch_cxtens}).

\begin{definition}\label{def:contfield}
A $C(X)$-algebra $A$ is called a  {\it continuous field} if for each $a \in A$, the mapping
\begin{equation*}
  X\ni x\xmapsto{\quad}\|a\|_x\in \bR_{\ge 0}
\end{equation*}
is continuous. 
\end{definition}

\section{Local-Triviality Dimension and Fields}\label{se:secont}

If $(C(\bG), \Delta)$ is a compact quantum group with a coaction on a unital $C^*$-algebra $A$, then \cite[Definition 3.1]{hajacindex_xv2} defines the weak and plain local-triviality dimensions as follows. Here, we let $t$ denote the identity function in $C_0((0, 1])$. We also note that since $C(\bG)$ denotes a compact {\it quantum} group, it need not be commutative as a $C^*$-algebra, despite the notation used.

\begin{definition}\label{def:LT}
The local-triviality dimension $\trivdim{A}{\bG}$ is the infimum of the set of $n \in \mathbb{N}$ such that there exist $\bG$-equivariant $*$-homomorphisms $\rho_0, \ldots, \rho_n: C_0((0, 1]) \otimes C(\bG) \to A$ satisfying $\sum\limits_{j=0}^n \rho_j(t \otimes 1) = 1$.
\end{definition}

\begin{definition}\label{def:WLT}
The weak local-triviality dimension $\wtrivdim{A}{\bG}$ is the infimum of the set of $n \in \mathbb{N}$ such that there exist $\bG$-equivariant $*$-homomorphisms $\rho_0, \ldots, \rho_n: C_0((0, 1]) \otimes C(\bG) \to A$ with $\sum\limits_{j=0}^n \rho_j(t \otimes 1)$ invertible.
\end{definition}

The strong variant of the local-triviality dimension, found in \cite[Definition 3.20]{hajacindex_xv2}, is defined in reference to the noncommutative join operation of \cite[Definition 4.1]{joinfusion} or \cite[Definition 1.4]{BDH}. Here $E_n^\Delta \bG$ denotes the join of $n+1$ copies of $\bG$.

\begin{definition}\label{def:SLT}
The strong local-triviality dimension is the infimum of the set of $n \in \mathbb{N}$ such that there exists a $\bG$-equivariant unital $*$-homomorphism $\rho: C(E_n^\Delta \bG) \to A$. 
\end{definition}

For actions of (non-quantum) groups $\bG$ on commutative unital $C^*$-algebras $A = C(X)$, the three dimensions coincide: for the weak and plain dimensions, there is a straightforward proof using conjugation, and equality of the plain and strong dimensions follows from \cite[Proposition 2.4]{cpt_dim}. However, if $A$ is noncommutative, then the dimensions may differ, even if $\bG$ is an ordinary group. Of particular note is the fact that the join of any number of copies of $\bZ/2$ is a sphere, and consequently, the strong local-triviality dimension has a direct connection to the Borsuk-Ulam Theorem. In fact, the local-triviality dimensions introduced in \cite{hajacindex_xv2} were developed in tandem with noncommutative Borsuk-Ulam theory, in the setting of free coactions. 

A coaction $\delta: A \to A \otimes_{\text{min}} C(\bG)$ is called \textit{free} (see \cite{ell} or \cite[(1.10)]{BDH}) if
\[ \overline{\text{span}_\bC\{(a \otimes 1)\delta(b): a, b \in A\}} \, = \, A \otimes_{\text{min}} C(\bG) .\]
Multiple characterizations of freeness are also found in \cite[Theorem 0.4]{freeness}. Moreover, finiteness of any one of the local-triviality dimensions implies freeness of the coaction by \cite[Theorem 3.8]{hajacindex_xv2}. We will primarily be interested in the case that $A$ is noncommutative, but $\bG$ is an ordinary group (in particular, a compact abelian group). In this case, there is an equivalent condition to freeness found in \cite[Definition 1.6]{rieffelproper} or \cite[Definition 7.1.4]{phillipsfreeness}. If $\bG$ is compact abelian and $\alpha$ is an action of $\bG$ on $A$, then for each character $\gamma \in \widehat{\bG}$, there is an associated eigenspace (or spectral subspace)
\begin{equation*}
  A_\gamma \, := \, \{a \in A\ :\ \alpha_g(a) = \gamma(g) a,\ \forall g\in \bG\}.
\end{equation*}
The action $\alpha$ is then free precisely when
\[ \overline{A_\gamma A A_\gamma^*} = A. \]

For the case that $\bG$ is a closed subgroup of $\bS^1$, \cite[Propositions 3.1 and 3.2]{cpt_dim} show how the local-triviality dimensions may each be recast in terms of elements of the spectral subspaces.

\begin{proposition}\label{pr:alt_dim_circle}
  Consider an action of $\bS^1$ on a unital $C^*$-algebra $A$, and let $\gamma$ be a generator of $\widehat{\bS^1} \cong \bZ$. Then
  \begin{equation*}
    \trivdim{A}{\bS^1} = \inf\left\{n \in \mathbb{N}: \exists \, \operatorname{ normal } \, a_0, \ldots, a_n \in A_\gamma, \, \sum\limits_{i=0}^n a_i a_i^* = 1\right\},
  \end{equation*} 
  \begin{equation*}
    \wtrivdim{A}{\bS^1} = \inf\left\{n \in \mathbb{N}: \exists \, \operatorname{ normal } \, a_0, \ldots, a_n \in A_\gamma, \, \sum\limits_{i=0}^n a_i a_i^*  \, \operatorname{ is } \, \operatorname{ invertible}\right\},
  \end{equation*} 
  and
  \begin{equation*}
    \strivdim{A}{\bS^1} = \inf\left\{n \in \mathbb{N}: \exists \, \operatorname{ normal} \, \operatorname{commuting} a_0, \ldots, a_n \in A_\gamma, \, \sum\limits_{i=0}^n a_i a_i^* = 1\right\}.
  \end{equation*}
\end{proposition}

For $\bZ/k$-actions, an additional spectral condition is needed. Set
\begin{equation*}
\stars{k} := \{te(m/k) \ :\ t\in [0,1],\ m \in \{0, 1, \ldots, k-1\} \}.
\end{equation*}
This \lq\lq asterisk\rq\rq\hspace{0pt} shape is star-convex, with center $0$, generated by the $k$th roots of unity.

\begin{proposition}\label{pr:alt_dim_Z/k}
  Consider an action of $\bZ/k$ on a unital $C^*$-algebra $A$, and let $\gamma$ be a generator of $\widehat{\bZ/k} \cong \bZ/k$. Then
  \begin{equation*}
    \trivdim{A}{\bZ/k} = \inf\left\{n \in \mathbb{N}: \exists \, \operatorname{ normal } \, a_0, \ldots, a_n \in A_\gamma, \, \sigma(a_i) \subseteq \stars{k}, \, \sum\limits_{i=0}^n a_i a_i^* = 1\right\},
  \end{equation*} 
  \begin{equation*}
    \wtrivdim{A}{\bZ/k} = \inf\left\{n \in \mathbb{N}: \exists \, \operatorname{ normal } \, a_0, \ldots, a_n \in A_\gamma, \, \sigma(a_i) \subseteq \stars{k}, \, \sum\limits_{i=0}^n a_i a_i^*  \, \operatorname{ is } \, \operatorname{ invertible}\right\},
  \end{equation*} 
  and
  \begin{equation*}
    \strivdim{A}{\bZ/k} = \inf\left\{n \in \mathbb{N}: \exists \, \operatorname{ normal} \, \operatorname{commuting} a_0, \ldots, a_n \in A_\gamma, \, \sigma(a_i) \subseteq \stars{k}, \, \sum\limits_{i=0}^n a_i a_i^* = 1\right\}.
  \end{equation*}
\end{proposition}

Since $\stars{2} = [-1, 1]$, the elements $a_i$ used to compute $\bZ/2$-dimensions are self-adjoints. As such, dimensions of $\bZ/2$-actions are typically much easier to compute than other cases. Also, as the set of invertible elements of $A$ is open, but $\{1\}$ is closed, the propositions suggest that weak and plain local-triviality dimensions may behave differently for fields of $C^*$-algebras. At the very least, for $\bG$-$C(X)$-algebras, we have the following immediate estimate on the global dimensions in terms of their respective fiber-wise counterparts.

\begin{lemma}\label{le:lesup}
  For a compact group $\bG$ and a $\bG$-$C(X)$-algebra $A$, we have
  \begin{equation}\label{eq:supax}
    \trivdim{A}{\bG}\ge \sup_{x\in X}\trivdim{A_x}{\bG},
  \end{equation}
  and similarly for $\wtrivdim{-}{\bG}$ and $\strivdim{-}{\bG}$.
\end{lemma}
\begin{proof}
  Immediate, as there exist $\bG$-equivariant morphisms $A\to A_x$ for all $x\in X$.  
\end{proof}

Cohomological obstructions can render the inequality of \Cref{le:lesup} strict.

\begin{example}\label{ex:supstrict}
  The group $\bG$ will be $\bZ/2$ and the space $X$ will be $\bP^1:=\bC\bP^1$, the complex projective line (i.e. space of lines in $\bC^2$). The global algebra $A$ underlying the field is simply
  \begin{equation*}
    A:=C(X)\otimes M_2\cong \mathrm{Cont}(X\to M_2).
  \end{equation*}
  The $\bZ/2$-action, on the other hand, is ``twisted'' conformant to the topology of $X$: the action on the fiber $A_x\cong M_2$ for a line
  \begin{equation*}
    \bP^1=X\ni x = \bC v,\quad v\in \bC^2\setminus\{0\}
  \end{equation*}
  is conjugation by the self-adjoint unitary $u_x:=2P_x-1$, where $P_x$ is the orthogonal projection on $x$. 

  The fiber-wise actions on the $A_x$ then have (weak, and plain, and strong) local-triviality dimension 0 by \cite[Proposition 4.5]{cpt_dim}, so the right-hand supremum of \Cref{eq:supax} vanishes. On the other hand, the vanishing of any of the dimensions for $A$ as a whole would require the existence of a single odd self-adjoint unitary
  \begin{equation*}
    v\in C(X)\otimes M_2. 
  \end{equation*}
  Each fiber $v_x\in A_x\cong M_2$ satisfies
  \begin{equation*}
    u_x v_x u_x^{-1} = -v_x,
  \end{equation*}
  so the $1$-eigenspace of $v_x$ will coincide with neither of the two eigenspaces of $u_x$ (with respective eigenvalues $\pm 1$). Identifying
  \begin{equation*}
    X=\bP^1\cong \bS^2,
  \end{equation*}
  said eigenspaces of $u_x$ are nothing but $x$ and its antipodal point $-x\in \bS^2$. The map
  \begin{equation*}
    \bS^2\cong \bP^1\ni x\xmapsto{\quad}\text{1-eigenspace of }v_x
  \end{equation*}
  is thus a self-map of the sphere that has no fixed points and sends no point to its antipode, and no such map exists by \cite[Corollary 1.24]{vick_hom_2e_1994}.
\end{example}

\begin{remark}\label{re:cp1}
  A somewhat different take on \Cref{ex:supstrict} will make it a simple matter to compute
  \begin{equation}\label{eq:p1dimis1}
    \wtrivdim{A}{\bG}
    =
    \wtrivdim{C(\bP^1)\otimes M_2}{\bZ/2}
    =
    1.
  \end{equation}
  Equip $\bP^1$ with its complex analytic structure or, equivalently for the purpose of manipulating holomorphic bundles, its complex projective algebraic structure (see \cite[\S 12]{ser_gaga}). The line bundle
  \begin{equation*}
    X=\bP^1\ni x
    \xmapsto{\quad}
    \text{$1$-eigenspace of }u_x
  \end{equation*}
  is the {\it universal subbundle} \cite[\S 3.2.3]{3264} on $\bP^1$ and hence what algebraic geometers denote by $\cO(-1)$ \cite[Definition preceding Proposition 5.12]{hrt}. Its orthogonal complement
  \begin{equation*}
    \bP^1\ni x
    \xmapsto{\quad}
    \text{$(-1)$-eigenspace of }u_x
  \end{equation*}
  in the trivial rank-2 bundle on $\bP^1$ is then $\cO(1)$, as the {\it degrees} \cite[Exercise II.6.12]{hrt} of the two bundles must add up to $0$. A choice of non-zero self-adjoint odd $v_x$ for $x$ ranging over some open $U\subseteq \bP^1$ means precisely an identification $\cO(-1)\cong \cO(1)$ over $U$, or a trivialization of
  \begin{equation*}
    \Hom(\cO(-1),\cO(1))|_U\cong \cO(2)|_U.
  \end{equation*}
  Such trivializations exist over {\it any} proper open subset of $\bP^1$ (e.g. the complement of a point), per the standard classification \cite[Theorem 2.1.1]{oss-vb} of line bundles on $\bP^1$. It follows that there are trivializations
  \begin{equation*}
    \cO(-2)|_{U_i}\cong \cO|_{U_i}
    ,\quad i=0,1,\quad
    U_i:=\bP^1\setminus\{p_i\}
    \quad\text{with}\quad p_0\ne p_1\in \bP^1
  \end{equation*}
  whose supports are respectively contained in $U_i$ and have interiors covering $\bP^1$. This gives the claimed \Cref{eq:p1dimis1}.
\end{remark}

A salient example of a $C(X)$-algebra is that of sections of an {\it $n\times n$ matrix-algebra bundle} $\cM$ on $X$ in the sense of \cite[Definition 18.1.3]{hjjm_bdle}: a vector bundle $\cM\to X$ with local trivializations $\cM|_U\cong U\times M_n$ respecting the $C^*$-algebra structures along the fibers. There will be no harm in allowing $n$ to vary, as $X$ simply decomposes into clopen subspaces over which $n$ is constant. 

Having to work with such trivializations extensively, we need some language to express the extent to which they exist. To that end, the following notions specialize the {\it (numerical) $\bG$-index} (\cite[Definition 3.13]{ziv_eqvrnt-1}, \cite[Definition 6.2.3]{mat_bu}) of a $\bG$-action on a space $X$, which is typically compact. That invariant can be recast via \cite[Theorem 3.4]{cf1} (where the term is {\it co-index} instead) as the smallest $n$ for which $X/\bG$ admits an open cover by $U_0, \ldots, U_n$, above which the action induces trivial principal $\bG$-bundles.

\begin{definition}\label{def:ix}
  \begin{enumerate}[(1),wide=0pt]
  \item\label{item:def:ix:vb} The {\it index} $\ix(\cF)$ of a vector bundle $\cF$ over a space $X$ (which is typically compact Hausdorff) is the smallest $m$ for which $X$ admits an open cover $X=\bigcup_{i=0}^m U_i$ trivializing each $\cF|_{U_i}$.

    More generally, for a family $(\cF_j)_j$ of vector bundles we write $\ix(\cF_j)_j$ for the smallest $m$ for which a cover as above trivializes {\it all} $\cF_j$ simultaneously.

  \item\label{item:def:ix:mb} In parallel, for a {\it matrix-algebra} bundle $\cM$ over $X$ the indices $\ix(\cM)$ or $\ix(\cM_j)_j$ are defined as in \Cref{item:def:ix:vb}, except that the trivializations $\cM_j|_{U_i}\cong U_i\times M_{n_j}$ are assumed to respect the $C^*$-algebra structures.
  \end{enumerate}
  To distinguish between the two notions of index we rely either on context (vector vs. matrix bundles) or, for added clarity, on subscripts: $\ix_{\cat{VB}}$ for \Cref{item:def:ix:vb} and $\ix_{\cat{MB}}$ for \Cref{item:def:ix:mb}. 
\end{definition}

\begin{remarks}\label{res:howspec}
  \begin{enumerate}[(1),wide=0pt]
  \item Parts \Cref{item:def:ix:vb} and \Cref{item:def:ix:mb} of \Cref{def:ix} are indeed instances of numerical $\bG$-indices (at least along the clopen subsets of $X$ where the ranks of the vector/matrix bundles involved are constant), with $\bG$ either a unitary or projective unitary group. As explained in \cite[Assertion 18.3.4]{hjjm_bdle}, rank-$n$ vector bundles and $n\times n$ matrix bundles ``are'' nothing but principal $U(n)$- and $PU(n)$-bundles, respectively.

  \item A $\bG$-action on a matrix bundle $\cM\to X$ induces one on $\Gamma(\cM)$ as a $C(X)$-algebra; conversely, a $\bG$-$C(X)$-algebra structure on $\Gamma(\cM)$ induces actions on the fibers $\cM_x = \Gamma(\cM)_x$ that glue together into an action on $\cM$. There is thus no harm in passing freely between the two types of action, as we will frequently in the sequel. 
  \end{enumerate}
\end{remarks}

We also need the following numerical invariants attached to finite-cyclic-group actions on matrix algebras or bundles.

\begin{definition}\label{def:quot}
  \begin{enumerate}[(1),wide=0pt]
  \item Consider an action $\bZ/k\times M_n\xrightarrow{\triangleright} M_n$ on the $C^*$-algebra $M_n$. By \cite[Example II.5.5.14]{blk}, a generator of $\bZ/k$ is necessarily  conjugation by a unitary $v\in M_n$, $v^k = I_n$, uniquely determined up to scaling by some power $\zeta^j$, $j\in \bZ/k$, of a primitive $k^{th}$ root of unity $\zeta$.

    The {\it quotient spread} of the action is
    \begin{equation}\label{eq:qt}
      \qt(\triangleright)
      :=
      \left\lceil
        \frac{\max\dim(v\text{-eigenspace})}{\min\dim(v\text{-eigenspace})}
      \right\rceil
    \end{equation}
    The aforementioned uniqueness of $v$ up to scaling makes the definition sound.

  \item Similarly, for a $\bZ/k$-action on a matrix bundle $\cM\to X$ or directly on a section $C(X)$-algebra $\Gamma(\cM)$, we have
    \begin{itemize}
    \item {\it fiber-wise quotient spreads} $\qt_x(\text{action})$ at $x\in X$, defined as in \Cref{eq:qt} for a unitary generator $v_x\in \Gamma(\cM)_x\cong M_{n_x}$ of the action;

    \item and a {\it global quotient spread} $\qt:=\sup_{x\in X}\qt_x$. 
    \end{itemize}
  \end{enumerate}
\end{definition}

A follow-up on \Cref{def:ix,def:quot}:

\begin{definition}\label{def:ixact}
  For a $\bZ/k$-action on a matrix bundle $\cM$ over $X$, the corresponding index $\ix(\cM,\triangleright) = \ix(\triangleright)$ is the smallest $m$ for which there is an open cover $X=\bigcup_{i=0}^{m} U_i$ with $\cM$ trivial, $\triangleright$ inner over each $U_i$ induced by some unitary $s_i\in \Gamma(\cM|_{U_i})$, $s_i^k=1$, with all eigenspaces of $s_i$ again trivial as bundles over $U_i$. 
\end{definition}

\begin{theorem}\label{th:zkmbact}
  Given a $\bZ/k$-action $\triangleright$ on a matrix bundle $\cM\to X$ over a compact Hausdorff space, we have
  \begin{equation}\label{eq:th:zkmbact:wineq}
    \qt - 1 \,\, \leq \,\, \wtrivdim{\Gamma(\cM)}{\bZ/k} \,\, \leq \,\,  \qt \cdot \left (\ix(\cM,\triangleright)+1\right) -1.
  \end{equation}
\end{theorem}
\begin{proof}
First consider when $X$ is a singleton, so $\ix(\cM, \triangleright) = 0$ and $\cM$ is a matrix algebra $M_n$. The result \cite[Theorem 4.1]{cpt_dim} effectively proves that $\wtrivdim{M_n}{\bZ/k} = \qt - 1$. Although that discussion focuses on {\it gauge} actions on matrix algebras, the arguments therein still work for arbitrary conjugation actions by unitaries $v\in M_n$, $v^k=1$. 

Applying this result to general bundles shows that each individual fiber $\Gamma(\cM)_x \cong M_{n_x}$ satisfies $\wtrivdim{\Gamma(\cM)_x}{\bZ/k} = \qt_x - 1$. Because there is an equivariant map
\[ \Gamma(\cM) \xrightarrow{\quad} \Gamma(\cM)_x \]
for each $x \in X$, we obtain that
\[ \wtrivdim{\Gamma(\cM)}{\bZ/k} \,\, \geq \,\, \sup_{x \in X} \, \wtrivdim{\Gamma(\cM)_x}{\bZ/k} \,\,  = \,\, \sup_{x \in X} \, \qt_x - 1  \,\, = \,\, \qt - 1, \]
which establishes the lower bound on the weak dimension.

For the upper bound on the weak dimension, we may assume $\qt$ is finite, as otherwise the result is trivial. Also, define $\ix := \ix(\cM, \triangleright)$ for simplicity, and consider an open cover $X=\bigcup_{i=0}^{\ix} U_i$ witnessing \Cref{def:ixact}. That is, the open cover trivializes the matrix bundle $\cM$ and renders the action inner, by unitary sections $v_i\in \Gamma(\cM|_{U_i})$ with $v_i^k=1$, and with trivial $v_i$-eigenbundles
  \begin{equation*}
    \cE_{i,j}:=\ker(\zeta^j-v_i), 
    \quad \quad
    0\le i\le \ix,
    \quad 
    j\in \bZ/k.
  \end{equation*}
If $i$ is fixed, the various $\cE_{i,j}$ are uniquely determined up to cycling through the $j$ indices. 

Because $\qt$ is assumed finite, each $\cE_{i,j}$ has nonzero vector space dimension at each $x \in X$. If $i$ is fixed and $x \in U_i$ is fixed, consider the smallest such eigenspace $W_0 := (\cE_{i,j_0})_x$. It is possible to define a partial isometry on $M_{n_x}$ that maps $W_0$ isometrically onto a subspace $W_1 \subseteq (\cE_{i,j_0+1})_x$ (which is guaranteed to be orthogonal to $W_0$ because the eigenspaces are orthogonal), maps $W_1$ isometrically onto some $W_2 \subseteq (\cE_{i,j_0+2})_x$, and so on, eventually returning to $W_k = W_0 = (\cE_{i,j_0})_x$. Further, map every element of $\left( \bigoplus_{m=0}^{k-1} W_m \right)^\perp$ to $0$. This procedure produces a partial isometry $\alpha$ in the generating spectral subspace of the action of $\bZ/k$ on this fiber. In fact, $\alpha$ is normal, as the initial space and final space are both equal to the direct sum $\bigoplus_{m=0}^{k-1} W_m$, and it also follows from the construction that $\alpha^k$ is equal to the projection onto that same direct sum. The subspaces $W_m$ for $m > 0$ may be proper subspaces of the corresponding eigenspaces, but because $\qt$ is defined as the (largest) quotient spread, we need at most $\qt$ such operators $\alpha_1, \ldots, \alpha_{\qt}$ to guarantee full coverage of each eigenspace. That is, $\sum\limits_{s=1}^{\qt} \alpha_s^* \alpha_s$ is an invertible element of  $M_{n_x}$.

Now, consider that since the eigenbundles $\cE_{i,j}$ are trivial on $U_i$, we may apply this construction on each $U_i$, rather than at an individual point. This produces normal partial isometries $\alpha_{i,s} \in \Gamma(\cM|_{U_i})$, for $0 \leq i \leq \ix$ and $1 \leq s \leq \qt$, that map
  \begin{equation*}
    \cE_{i,j}
    \xrightarrow{\quad \alpha_{i,s}\quad}
    \cE_{i,j+1}
    ,\quad
    j\in \bZ/k
  \end{equation*}
and satisfy the conditions that each $\alpha_{i,s}^k$ is a projection and $\sum\limits_{s=1}^{\qt} \alpha_{i,s}^* \alpha_{i,s}$ is invertible on each $U_i$. As before, we have that the $\alpha_{i,s}$ are in the generating spectral subspace of the $\bZ/k$-action by design. However, each $\alpha_{i,s}$ is only a section on $U_i$, not the full space $X$.

Finally, we may fix a refinement $(V_i)_{i=0}^{\ix}$ of the cover $(U_i)_{i=0}^{\ix}$, with
  \begin{equation*}
    V_i = \accentset{\circ}{V_i}\subseteq \overline{V_i}\subseteq U_i.
  \end{equation*}  
A partition of unity argument (see \cite[Theorem 36.1]{mnk}) allows us to rescale the operators $\alpha_{i,s} \in \Gamma(\cM|_{U_i})$ into global sections $\phi_i^{1/2} \alpha_{i,s}\in \Gamma(M)$, where the $\phi_i$ are nonnegative scalar-valued continuous functions on $X$, supported in $U_i$, with $\phi_i|_{V_i}\equiv 1$ and $\sum\limits_{i = 0}^{\ix} \phi_i \equiv 1$. Note that each $\phi_i$ is a central element of $\Gamma(\cM)$ and is fixed by the action of $\bZ/k$, so multiplying by $\phi_i^{1/2}$ has not removed any of the key properties of $\alpha_{i,s}$. Because $1 \leq s \leq \qt$ and $0 \leq i \leq \ix$, we have produced a list of $\qt\cdot(\ix + 1)$ normal operators $\alpha_{i,s}$ in the generating spectral subspace of the $\bZ/k$-action such that $\sum\limits_{i,s} \alpha_{i,s}^*\alpha_{i,s}$ is (globally) invertible. By \Cref{pr:alt_dim_Z/k}, we have that 
\[ \wtrivdim{\Gamma(\cM)}{\bZ/k} \,\, \leq \,\,  \qt \cdot \left (\ix+1\right) -1.\]

\end{proof}

Just as \cite[Theorem 4.1]{cpt_dim} implies the freeness criterion of \cite[Corollary 4.2]{cpt_dim} for gauge actions, \Cref{th:zkmbact} delivers a matrix-bundle analogue (and generalization).

\begin{corollary}\label{cor:wdimfinfree}
  For a $\bZ/k$-action $\triangleright$ on a matrix bundle $\cM\to X$ over a compact Hausdorff space, the following conditions are equivalent.
  \begin{enumerate}[(a)]  

  \item\label{item:cor:wdimfinfree:nontriv} At every $x$, a unitary $v_x$, $v_x^k=1$ inducing the $\bZ/k$-action has full spectrum $\bZ/k\subset \bS^1$. 
    
  \item\label{item:cor:wdimfinfree:wtrivfin} $\wtrivdim{\Gamma(\cM)}{\bZ/k}<\infty$.  

  \item\label{item:cor:wdimfinfree:free} The action induced by $\triangleright$ on $\Gamma(\cM)$ is free.    
  \end{enumerate}
\end{corollary}
\begin{proof}
  The implication \Cref{item:cor:wdimfinfree:nontriv} $\Longrightarrow$ \Cref{item:cor:wdimfinfree:wtrivfin} follows from the upper bound \Cref{eq:th:zkmbact:wineq} on $\wtrivdim{\Gamma(\cM)}{\bZ/k}$, where we note that $\ix(\cM, \triangleright)$ is always finite, and the full spectrum condition of \Cref{item:cor:wdimfinfree:nontriv} guarantees that the definition of $\qt$ does not involve dividing by zero. Moreover, \Cref{item:cor:wdimfinfree:wtrivfin} $\Longrightarrow$ \Cref{item:cor:wdimfinfree:free} follows from the general result \cite[Theorem 3.8]{hajacindex_xv2}. As for \Cref{item:cor:wdimfinfree:free} $\Longrightarrow$ \Cref{item:cor:wdimfinfree:nontriv}, freeness means by \cite[Theorem 8.1.7]{montg_halg} that the $\bZ/k$-grading on $A:=\Gamma(\cM)$ induced by the action is {\it strong}:
  \begin{equation*}
    A_{\zeta^j} A_{\zeta^{-j}} = A_1,
    \quad\forall j\in \bZ/k.
  \end{equation*}
  It is enough to focus on a single fiber at $x\in X$: the claim is that the $\bZ/k$-grading induced by the conjugation action of a unitary $v\in M_n$ with $v^k=1$ is strong precisely when the $\zeta^j$-eigenspaces of $v$ are all non-zero for $j\in \bZ/k$. This is easily seen: if, say,
  \begin{equation*}
    \ker(\zeta^j-u) = \{0\} \ne \ker(\zeta^{j+1}-u),
  \end{equation*}
  then the orthogonal projection onto the latter (non-zero) space annihilates $M_{n,\zeta}M_{n,\zeta^{-1}}$.
\end{proof}

As for {\it plain} LT dimensions, \Cref{th:zkmbact-s} below is to \cite[Theorem 4.4]{cpt_dim} what \Cref{th:zkmbact} is to \cite[Theorem 4.1]{cpt_dim}. This is in the same spirit of ``spreading'' the results of \cite[\S 4]{cpt_dim} over a base space $X$.

\begin{theorem}\label{th:zkmbact-s}
  For a $\bZ/k$-action $\triangleright$ on a matrix bundle $\cM\to X$ over a compact Hausdorff space, the following conditions are equivalent.

  \begin{enumerate}[(a)]
  \item\label{item:th:zkmbact-s:specbd} $\trivdim{\Gamma(\cM)}{\bZ/k}\le \ix(\cM,\triangleright)$.

  \item\label{item:th:zkmbact-s:infty} $\trivdim{\Gamma(\cM)}{\bZ/k}<\infty$.

  \item\label{item:th:zkmbact-s:equidim} At every $x\in X$, a unitary $v_x\in \cM_x$ with $v_x^k=1$ inducing the action has equidimensional $\zeta^j$-eigenspaces for $j\in \bZ/k$. 
  \end{enumerate}
\end{theorem}
\begin{proof}
  The implication \Cref{item:th:zkmbact-s:specbd} $\Longrightarrow$ \Cref{item:th:zkmbact-s:infty} is formal, and \Cref{item:th:zkmbact-s:infty} $\Longrightarrow$ \Cref{item:th:zkmbact-s:equidim} follows from the proof of \cite[Theorem 4.4]{cpt_dim}, which provides the $X=\{*\}$ version. We thus turn to \Cref{item:th:zkmbact-s:equidim} $\Longrightarrow$ \Cref{item:th:zkmbact-s:specbd}. 

  As in the proof of \Cref{th:zkmbact}, consider an open cover $X=\bigcup_{i=0}^{\ix}U_i$ trivializing $\cM$, generators $v_i$ for the $\bZ/k$-action restricted to $U_i$, and the eigenspaces $\cE_{i,j}$, $j\in \bZ/k$ of $v_i$. We can construct operators $\alpha_{i,s}$ as in the earlier proof, with the only difference being that now we can select a {\it single}, {\it unitary} $\alpha_i \in \Gamma(\cM|_{U_i})$ for each $i$. The same partition of unity argument will produce a list of $\ix(\cM,\triangleright) + 1$ operators in the vein of \Cref{pr:alt_dim_Z/k}, demonstrating that $\trivdim{\Gamma(\cM)}{\bZ/k} \leq \ix(\cM,\triangleright)$.
\end{proof}

The case $X=\{*\}$ of \Cref{th:zkmbact-s}, treated in \cite[Theorem 4.4]{cpt_dim}, is not sufficiently refined to distinguish between ordinary and strong LT dimensions. Contrast this to \Cref{ex:supstrict}, which meets the requirements of \Cref{th:zkmbact-s} and hence has finite LT (so also weak LT) dimension: 


\begin{proposition}\label{pr:s2infsdim}
  For the $\bZ/2$-action on $A:=C(\bP^1)\otimes M_2$ of \Cref{ex:supstrict} we have
  \begin{equation*}
    \wtrivdim{A}{\bZ/2} = \trivdim{A}{\bZ/2} = 1
    \quad\text{and}\quad
    \strivdim{A}{\bZ/2}=\infty.
  \end{equation*}
\end{proposition}
\begin{proof}
  \Cref{ex:supstrict} already argues that none of the dimensions can vanish, and the upper bound of $1$ for the ordinary (hence also weak) LT dimension follows from \Cref{th:zkmbact-s} item \Cref{item:th:zkmbact-s:specbd} by the selfsame example: $\ix(\cM,\triangleright)=1$, since a cover by two punctures of $\bP^1$ trivializes everything in sight. 
  
  As for the infinitude of the strong LT dimension, recall first that the action at $x\in \bP^1$ is conjugation by the self-adjoint (hence involutive) unitary having $x$ (regarded as a line in $\bC^2$) and the orthogonal complement $x^{\perp}\le \bC^2$ as its $1$- and $(-1)$-eigenspaces, respectively.

  Consider commuting self-adjoint odd elements
  \begin{equation*}
    a_i\in A\cong C(\bP^1, M_2)
  \end{equation*}
  witnessing finite strong LT dimension. At each $x\in \bP^1$, at least one $a_i$ will be non-vanishing, and the commutation assumption ensures that at every $x$, the unordered pair of mutually-orthogonal $a_i(x)$-eigenspaces (corresponding to the positive and negative eigenvalue of $a_i$ respectively) is $i$-independent so long as $a_i(x)\ne 0$.

  We thus have a continuous map
  \begin{equation}\label{eq:xtopair}
    \begin{aligned}
      \bP^1\cong \bS^2
      \quad\ni\quad
      x
      \xmapsto{\quad}
      \big(\text{pair of $a_{i}(x)$-eigenspaces}\big)
      \quad&\in\quad
             \bS^2/(\text{antipodal $\bZ/2$})\\
           &\cong
             \quad\text{real projective plane }\bR\bP^2.
    \end{aligned}
  \end{equation}
  Because (by oddness) each eigenspace of $a_i$ is orthogonal to each eigenspace of $u$, \Cref{eq:xtopair} coincides with the canonical surjection $\bS^2\to \bR\bP^2$ {\it nowhere}. A continuous lift $\bS^2\to \bS^2$ thereof (which does exist by \cite[Lemma 79.1]{mnk}, as $\bS^2$ is simply-connected) will thus have no coincidences with either the identity or the antipodal map, once more violating \cite[Corollary 1.24]{vick_hom_2e_1994}.
\end{proof}

The following semicontinuity result for the weak local-triviality dimension of $\bZ/2^m$-actions applies whether or not a given field is continuous.

\begin{theorem}\label{th:semicont}
  Let $A$ be a unital $(\bZ/2^m)$-$C(X)$-algebra for some $m \in \bZ^+$. The function
  \begin{equation*}
    X\ni x\xmapsto{\quad} \wtrivdim{A_x}{\bZ/2^m}\in \bZ_{\ge 0}\sqcup \{\infty\}
  \end{equation*}
  is upper semicontinuous.
\end{theorem}
\begin{proof}
  We have to argue that for every $n \in \bZ_{\ge 0}$, the set
  \begin{equation*}
    \left\{x\in X\ |\ \wtrivdim{A_x}{\bZ/2^m}\le n\right\}\subseteq X
  \end{equation*}
  is open. Having fixed a point $x$ with $\wtrivdim{A_x}{\bZ/2^m}\le n$, consider the morphisms given by \Cref{def:WLT}:
  \begin{equation*}
    C_0((0,1])\otimes C(\bZ/2^m)\xrightarrow{\quad\rho_i\quad} A_x,\quad 0\le i\le n
  \end{equation*}
  with $\sum\limits_{i=0}^n\rho_i(t\otimes 1)$ invertible. The $\rho_i$ lift equivariantly to
  \begin{equation*}
    C_0((0,1])\otimes C(\bZ/2^m)\xrightarrow{\quad\overline{\rho}_i\quad} A
  \end{equation*}
  by {\it projectivity}, as in the proof of \cite[Proposition 5.4]{cpt_dim}, via \cite[Proposition 2.10]{Semiprojective}. Define $a := \sum\limits_{i=0}^n\overline{\rho}_i(t \otimes 1)$. Note that $\pi_x(a)$ is invertible, so we may choose an element $b \in A$ such that $\pi_x(b) = (\pi_x(a))^{-1}$. That is, both $\|ab - 1\|_x = 0$ and $\|ba - 1\|_x = 0$. By \cite[Proposition 1.2]{rief_flds}, for each $c \in A$ the mapping 
\begin{equation*} X \ni x \mapsto \|c\|_x
\end{equation*}
is upper semicontinuous, so there exists a neighborhood $U$ of $x$ such that for each $y \in U$, both $\|ab - 1\|_y < 1$ and $\|ba - 1\|_y < 1$. In particular, $\pi_y(a)$ is invertible for all $y \in U$. Since $\sum\limits_{i=0}^n\overline{\rho}_i(t \otimes 1)$ is invertible modulo $J_y$ for all $y \in U$, this shows that
\begin{equation*}
    \wtrivdim{A_{y}}{\bZ/2^m}\le n,\ \forall y \in U,
  \end{equation*}
  finishing the proof.
\end{proof}


The following examples show that the weak local-triviality dimension is not necessarily continuous, even if a given field of $C^*$-algebras is continuous in the sense of \Cref{def:contfield}.

\begin{example}\label{ex:ws3}
  Consider the antipodal $\bZ/2$-action on the $\theta$-deformed spheres $C(\mathbb{S}^3_\theta)$, which form a continuous field. The Borsuk-Ulam theorem implies that $\wtrivdim{\thetasphere{3}{}}{\bZ/2} = 3$. On the other hand, for $\theta\not\in\bZ$ the proof of \cite[Theorem 2.10]{bentheta} shows that if $\thetasphere{3}{\theta}$ is noncommutative, with $z_j = x_j + i y_j$, then
\begin{equation*}
  (x_1 + x_2)^2 + (y_1 + y_2)^2 = 1 + w
\end{equation*}
where $\|w\| < 1$. Hence, we have found two self-adjoint odd elements whose sum-square is invertible, and $\wtrivdim{\thetasphere{3}{\theta}}{\bZ/2} \leq 1$. The dimension is not zero, as there exists an equivariant quotient $\thetasphere{3}{\theta} \to \thetasphere{1}{}$, so $\wtrivdim{\thetasphere{3}{\theta}}{\bZ/2} = 1$. 

We conclude that the weak local-triviality dimension function is upper semicontinuous, but not continuous, in the parameter $\theta$.
\end{example}

\begin{example}\label{ex:defquant}
  Consider the {\it Berezin quantization} of the coadjoint orbit
  \begin{equation*}
    SU(2) / \bS^1
  \end{equation*}
  where $\bS^1$ is a maximal torus: see \cite[Theorems 1 and 2]{lands_sq} and \cite[Theorem III.1.11.1]{lands_topic}, for $G=SU(2)$, as recalled in \cite[Introduction]{rief_fuzz-conv-1}.
  
  The field $A$ of $C^*$-algebras (and actions) we consider is based on
  \begin{equation*}
    X := \bZ_{\ge 1}\sqcup\{\infty\}, \quad n\to\infty,
  \end{equation*}
  with
  \begin{itemize}
  \item the fiber $A_n$ at $n\in \bZ_{\ge 1}$ isomorphic to the matrix algebra $M_n$;
  \item the fiber $A_{\infty}$ isomorphic to $C(\bS^2)$;
  \item the action at $\infty$ given by translation of $SU(2)$ on $\bS^2\cong SU(2)/\bS^1$ or, equivalently, the usual action of $SU(2)$ by rotations, via the surjection
    \begin{equation}\label{eq:su2so3}
      SU(2)\xrightarrow{\quad}PSU(2)\cong SO(3);
    \end{equation}
  \item the action of $SU(2)$ on $M_n$, $n\in\bZ_{\ge 1}$ given by the conjugation action on
    \begin{equation*}
      M_n\cong B(\bC^n)
    \end{equation*}
    resulting from realizing $\bC^n$ as the irreducible $n$-dimensional representation of $SU(2)$ (see  \cite[Theorem VIII.4.2]{simon_fincpct}).    
  \end{itemize}  

  The compact group $SU(2)$ acts on everything in sight (i.e. all $A_x$, $x\in X$) via its quotient \Cref{eq:su2so3}. One can then restrict said actions to subgroups of $PSU(2)$. Consider, specifically, the subgroup $\bZ/2$ generated by the image in $PSU(2)$ of
  $\sigma:=
  \begin{pmatrix}
    0&-1\\
    1&\phantom{-}0
  \end{pmatrix}$.
  Because every rotation has fixed points on the sphere, no $\bZ/2$-action on $\bS^2$ obtained in this manner can be free; this means in particular that $\wtrivdim{A_{\infty}}{\bZ/2}=\infty$.
  
  On the other hand, $\sigma$ acts on $\bC^n$ with eigenvalues $\pm\lambda$, with $\lambda=1$ or $\lambda = i$ depending on the parity of $n$, and with the $(\pm\lambda)$-eigenspaces of dimension as close to equal as possible. It follows that $\sigma$ induces the usual gauge $(\bZ/2)$-action on $M_n$ resulting from the latter's realization as the graph algebra of a length-$n$ directed path. Finally, \cite[Proposition 4.5]{cpt_dim} shows that
  \begin{equation*}
    n \text{ is even } \,\, \implies \,\, \wtrivdim{A_n}{\bZ/2}=0,
  \end{equation*}
  which demonstrates the claimed discontinuity at $\infty$.

  This example applies to all three variants of the local-triviality dimension, so none of them are continuous over fields.
\end{example}

Another pathology can occur upon dropping some of the constraints on the group $\bG$ of \Cref{th:semicont}: there exist continuous fields for which even the weak local-triviality dimension fails to be upper semicontinuous.

\begin{example}\label{ex:nctori}
  Consider the noncommutative 2-tori $A_{\theta}=A_{\theta}^2$, the universal $C^*$-algebras respectively generated by two unitaries $u := u_1$ and $v := u_2$ subject to
  \begin{equation*}
    vu = e(\theta)uv. 
  \end{equation*}
  The $A_\theta$ are the fibers of a continuous field by \cite[Theorem 8.13 and Example 10.2]{rief_def-rd}, and all admit {\it ergodic} actions (see \cite[p.537]{gvf_ncg}) by the ordinary 2-torus $\bT^2\cong (\bS^1)^2$. This constitutes a continuous field of actions in the sense of \cite[Definition 3.1]{rief_flds}. Note the following.
  \begin{itemize}
  \item At $\theta=0$, we have $A_{0}\cong C(\bT^2)$ with the usual translation action, so that all dimensions
    \begin{equation*}
      \wtrivdim{A_0}{\bT^2} = \trivdim{A_0}{\bT^2} = \strivdim{A_0}{\bT^2} = 0
    \end{equation*}
    vanish by \cite[Remark 2.5]{cpt_dim}.

  \item On the other hand, for non-integral values of $\theta$, none of the dimensions can vanish. The $\bT^2$-action is ergodic, so the eigenspaces for the characters of $\bT^2$ are all 1-dimensional (see \cite[Theorem 1]{wass_ergd-1}). Therefore an equivariant morphism
    \begin{equation*}
      C(\bT^2)\xrightarrow{\quad}A_{\theta}
    \end{equation*}
    would have to be onto, which would force $A_{\theta}$ to be commutative.
  \end{itemize}
  Because all of the local-triviality dimensions vanish at $\theta=0$ but are {\it larger} everywhere else in a neighborhood of that point, the upper semicontinuity of \Cref{th:semicont} fails in this case.
\end{example}

Incidentally, in the context of \Cref{ex:nctori} we can say much more: not only are the LT dimensions strictly positive for non-integral $\theta$, they are infinite.

\begin{theorem}\label{th:ergdim0}
  Let $\bG$ be a compact quantum group acting ergodically on a unital $C^*$-algebra $A$. Consider the following conditions.
  \begin{enumerate}[(a)]
  \item\label{item:th:ergdim0:dims0} All of the LT dimensions vanish.
  \item\label{item:th:ergdim0:dimsfin} All of the LT dimensions are finite.
  \item\label{item:gradedwtrivfin} $\wtrivdim{A}{\bG}$ is finite.

  \item\label{item:th:ergdim0:factor} There is a factorization
    \begin{equation*}
      C(\bG)\xrightarrowdbl{\quad}A\xrightarrowdbl{\quad}C_r(\bG),
    \end{equation*}
    where the left-hand surjection induces the $\bG$-action on $A$. 
    

  \end{enumerate}
  We then have
  \begin{equation}\label{eq:genimpl}
    \text{\Cref{item:th:ergdim0:dims0}}
    \xLeftrightarrow{\quad}
    \text{\Cref{item:th:ergdim0:dimsfin}}
    \xLeftrightarrow{\quad}
    \text{\Cref{item:gradedwtrivfin}}
    \xLeftarrow{\quad}
    \text{\Cref{item:th:ergdim0:factor}},
  \end{equation}
  with full equivalence when $\bG$ is {\it of Kac type} \emph{\cite[post Proposition 1.7.9]{nt_cqg}}. 
\end{theorem}
\begin{proof}
  Among \Cref{eq:genimpl}, the only nontrivial implication is \Cref{item:gradedwtrivfin} $\Longrightarrow$ \Cref{item:th:ergdim0:dims0}. Consider morphisms
  \begin{equation}\label{eq:phii}
    C_0((0, 1])\otimes C(\bG)
    \xrightarrow{\quad\varphi_i\quad}
    A,\quad
    0\le i\le n
  \end{equation}
  witnessing the finiteness of the weak LT dimension, as in \Cref{def:WLT}. In particular, $\sum\limits_{j=0}^n \varphi_i(t \otimes 1)$ is invertible, so we can fix some $\varphi:=\varphi_{i_0}$ with $\varphi(t\otimes 1)\ne 0$.

  Recall from \cite[Corollary 4.1]{WZcpoz} that $\varphi$ uniquely retrieves a completely positive contractive {\it order-zero} \cite[Definition 2.3]{WZcpoz} map
  \begin{equation*}
    C(\bG)\ni x
    \xmapsto{\quad\widetilde{\varphi}\quad}
    \varphi(t\otimes x)
    \in
    A.
  \end{equation*}
  Note that $\widetilde{\varphi}$ is $\bG$-equivariant in the usual sense, and by \cite[Theorem 3.3]{WZcpoz} it is a scaling of a morphism $C(\bG)\to A$ by the $\bG$-invariant, positive, non-zero element $h:=\widetilde{\varphi}(1)\in A$. Since ergodicity of the action implies that $h$ is a scalar, we can divide by it to recover a $\bG$-equivariant $C^*$-morphism $C(\bG)\to A$.

  Observe also that a $\bG$-equivariant $C^*$-morphism is by necessity one-to-one on the dense Hopf $*$-algebra \cite[\S 1.6]{nt_cqg} $\cO(\bG)\le C(\bG)$: the $\bG$-invariant state
  \begin{equation*}
    \begin{tikzpicture}[auto,baseline=(current  bounding  box.center)]
      \path[anchor=base] 
      (0,0) node (l) {$A$}
      +(2,.5) node (u) {$A\otimes C(\bG)$}
      +(4,0) node (r) {$\bC$}
      ;
      \draw[->] (l) to[bend left=6] node[pos=.5,auto] {$\scriptstyle \text{coaction}$} (u);
      \draw[->] (u) to[bend left=6] node[pos=.5,auto] {$\scriptstyle \psi\otimes h_{\bG}$} (r);
      \draw[->] (l) to[bend right=6] node[pos=.5,auto,swap] {$\scriptstyle $} (r);
    \end{tikzpicture}
    \quad
    \psi:=\text{state on }A
    ,\quad
    h_{\bG}:=\text{the Haar state on $\bG$}
  \end{equation*}
  restricts to $h_{\bG}$ along the assumed morphism $C(\bG)\to A$, and $h_{\bG}$ is faithful on $\cO(\bG)$ \cite[Corollary 1.7.5]{nt_cqg}. If the action is 0-dimensional, then, the multiplicity of a $\bG$-representation $\rho$ in $A$ is precisely $\dim\rho$. For ergodic actions of Kac groups this is also an upper bound for that multiplicity \cite[Theorem 17]{boca_erg} (see also \cite[Theorem 1]{wass_ergd-1} for classical compact groups), hence the implication \Cref{item:th:ergdim0:dims0} $\Longrightarrow$ \Cref{item:th:ergdim0:factor} for Kac-type $\bG$. 
\end{proof}

In particular, taking $\bG:=\bT^n$ for $n\ge 2$, the field
\begin{equation*}
  \theta\xmapsto{\quad} C(\bT^n_{\theta})
\end{equation*}
of ergodic $\bT^n$-actions in \Cref{ex:nctori} defined over the skew-symmetric matrices $\theta$ has local-triviality dimension equal to $0$ precisely when $\theta$ is integral. Moreover, it has local-triviality dimension $\infty$ elsewhere. This proves a stronger claim than used in \Cref{ex:nctori}.

\Cref{th:ergdim0} will also allow us to revisit the matter of how the {\it freeness} of an action relates to finiteness of the local-triviality dimensions. Recall the following:
\begin{itemize}
\item if the weak local-triviality dimension (or rather, any one of the local-triviality dimensions) of an action is finite, then \cite[Theorem 3.8]{hajacindex_xv2} shows that the action is free;
  
\item for the gauge $\bZ/k$-actions on finite-acyclic-graph algebras, \cite[Corollary 4.2]{cpt_dim} shows that freeness is {\it equivalent} to finiteness of the weak local-triviality dimension;

\item however, freeness of an action of $\bZ/k$ on a matrix algebra $M_n$ does {\it not} imply the finiteness of $\trivdim{M_n}{\bZ/k}$ or $\strivdim{M_n}{\bZ/k}$, as seen in \cite[Proposition 4.5 and Remark 4.6]{cpt_dim}.
\end{itemize}

\Cref{th:ergdim0} allows us to demonstrate another key fact: freeness of an action does not even imply finiteness of the {\it weak} local-triviality dimension.

\begin{example}\label{ex:nontrivcocycle}
  The group $\bG:=(\bZ/q)^2$ acts ergodically on $M_q$, with the two generators operating as conjugations by
  \begin{equation*}
    \begin{pmatrix}
      1&&&&\\
      &\zeta&&&\\
       &&\ddots&&\\
       & &&\zeta^{q-2}&\\
      &&&&\zeta^{q-1}
    \end{pmatrix}
    \quad\text{and}\quad
    \begin{pmatrix}
      0&&&&1\\
      1&0&&&\\
       &1&\ddots&&\\
       & &\ddots&0&\\
      0&&&1&0
    \end{pmatrix}
  \end{equation*}
  respectively for some primitive $q^{th}$ root of unity $\zeta$ (see \cite[proof of Proposition 12.2, (12.7)]{gvf_ncg}). This in fact realizes $M_q$ as a {\it cocycle twist} \cite[\S 2.2]{el_morita-def} $C^*(\Gamma,\sigma)$ of the group algebra of
  \begin{equation*}
    \Gamma:=\widehat{\bG}\cong (\bZ/q)^2,
  \end{equation*}
  for a cocycle whose class generates
  \begin{equation*}
    H^2\left((\bZ/q)^2,\ \bC^{\times}\right)\cong \bZ/q
  \end{equation*}
by \cite[Theorem 9.74 and Theorem following Example 9.65]{rot_homolog}. Next, \cite[Theorem 8.1.7]{montg_halg} once more shows that the induced grading by the Pontryagin-dual group $\Gamma$ is  strong:
\begin{equation*}
  M_{q,\gamma} M_{q,\gamma^{-1}} = M_{q,1}=\bC,
  \quad\forall \gamma\in \Gamma,
\end{equation*}
since every homogeneous component $M_{q,\gamma}$ contains unitaries. Therefore, the action is free. However, \Cref{th:ergdim0} shows that the weak local-triviality dimension is infinite. 
\end{example}

\section{Rational tori and spheres}\label{se:rat}

In this section, we consider the local-triviality dimensions of $\theta$-spheres, with an emphasis on the rational case. The $\theta$-spheres and quantum tori are related by a polar decomposition that mirrors that of the commutative case. Consider, for example, a sphere $\bS^{2n-1}$ whose complex coordinates $z_j$ are written in polar form
\begin{equation*}
  z_j = t_j u_j, \,\,\,\,\, t_j \in [0, 1], u_j \in \bS^1.
\end{equation*}
For $f \in C(\bS^{2n-1})$, if choices of $t_1, \ldots, t_n$ are fixed, then we recover a function $g \in C(\bT^n)$ by
\begin{equation*}
  g(u_1, \ldots, u_n) := f(t_1u_1, \ldots, t_n u_n).
\end{equation*}
However, if some particular $t_j$ equals $0$, then $z_j = t_j u_j = 0$ is never changing, so the function $g$ must not depend on $u_j$. That is, $g$ is in the $C^*$-algebra generated by the other coordinate functions. This perspective extends to the $\theta$-spheres and quantum tori, as in \cite[Theorem 2.5]{no_sph}. We recall this result and introduce some notation below.

\begin{notation}\label{not:subsplx}
  Let $n\ge 2$ and $[n]:=\{1,\ \ldots,\ n\}$.
  \begin{enumerate}[(1)]

  \item Following \cite[\S 2]{no_sph}, write
    \begin{equation*}
      \bS^{n-1}_+:=\{(t_1,\ \ldots,\ t_{n})\in \bS^{n-1}\ |\ t_i\ge 0\} 
    \end{equation*}
    for the ``first-quadrant'' portion of the $(n-1)$-sphere.

  \item For any subset $F\subseteq [n]$ set
    \begin{equation*}
      \bS^{n-1}_{++,F}:=\left\{(t_i)\in \bS^{n-1}_+\ |\ t_i>0\text{ exactly when }i\in F\right\}.
    \end{equation*}

    Each $\bS^{n-1}_{++,F}$ is the interior of an $(|F|-1)$-dimensional simplex, and the closed simplex $\bS^{n-1}_+$ is the disjoint union of the various open simplices $\bS^{n-1}_{++,F}$. 

  \item Set
    \begin{equation*}
      \begin{aligned}
        \bS^{n-1}_{+,F} \,\, &:= \,\, \overline{\bS^{n-1}_{++,F}}\\
                       &= \,\, \left\{(t_i)\in \bS^{n-1}_+\ | \,\, \forall i\not\in F, \ t_i=0 \, \right\}.
      \end{aligned}      
    \end{equation*}
    
  \end{enumerate}

For a skew-symmetric matrix $\theta\in M_n(\bR)$, recall from \cite[Theorem 2.5]{no_sph} that
\begin{equation}\label{eq:contpart}
  C(\bS_{\theta}^{2n-1})
  \, \cong \,
  \mathrm{Cont}_{\partial}\left(\bS_+^{n-1}\xrightarrow{\quad}C(\bT^n_{\theta})\right),
\end{equation}
where the $\partial$ subscript indicates satisfaction of the following boundary condition: whenever some $t_i$ vanish, the continuous function
\begin{equation*}
  \bS_+^{n-1}\xrightarrow{\quad f\quad}C(\bT^n_{\theta})
\end{equation*}
is to take values in the $C^*$-subalgebra of $C(\bT^n_{\theta})$ generated by the unitaries $u_j$, $j\ne i$. That is, for all $F \subseteq [n]$, 
\begin{equation}\label{eq:polar}
  f\left(\bS^{n-1}_{+,F}\right)
  \subseteq
  C^*(\{u_j\ |\ j \in F\}).
\end{equation}
\end{notation}

In particular, morphisms into a $\theta$-sphere relate to (continuous) families of morphisms into a quantum torus, which will allow us to study their dimensions more carefully.

\begin{lemma}\label{le:wpdimthetasphle}
Every $\theta$-sphere has $\wtrivdim{\thetasphere{k}{\theta}}{\bZ/2}
    \le
    \trivdim{\thetasphere{k}{\theta}}{\bZ/2}
    \le k$.
\end{lemma}
\begin{proof}
For $k = 2n-1$, each generator $z_j = x_j + i y_j$ of $\thetasphere{2n-1}{\theta}$ is normal and odd, so the identity $1 = \sum\limits_{j=1}^n z_jz_j^* = \sum\limits_{j=1}^n x_j^2 + y_j^2$ shows that the weak and plain local-triviality dimensions are at most $2n-1$. A similar argument applies to $k = 2n-2$, with the final generator $z_n = x_n$ self-adjoint.
\end{proof}

\begin{theorem}\label{th:strongtheta}
 Every $\theta$-sphere has $\strivdim{\thetasphere{k}{\theta}}{\mathbb{Z}/2} \geq k$. If each entry of $\theta$ is a rational number with odd denominator, then in fact $\strivdim{\thetasphere{k}{\theta}}{\mathbb{Z}/2} = k$.
\end{theorem}

\begin{proof}
  The noncommutative Borsuk-Ulam theorem \cite[Corollary 3.8]{bentheta} implies that if $m < k$, there is no $\mathbb{Z}/2$-equivariant morphism $\thetasphere{m}{\phi} \to \thetasphere{k}{\theta}$. This applies in particular to the case when the domain sphere is commutative. Nonexistence of an equivariant morphism $\thetasphere{m}{} \to \thetasphere{k}{\theta}$ for $m < k$ is equivalent to the claim $\strivdim{\thetasphere{k}{\theta}}{\mathbb{Z}/2} \geq k$.
  
  Now, assume that each entry of $\theta$ is a rational number with odd denominator, and suppose $k = 2n - 1$ is odd. There exists an odd positive integer $M$ such that in the associated quantum torus $C(\mathbb{T}^n_\theta)$, the elements $u_1^M, \ldots, u_n^M$ are central. The functions $f_i: \bS^{n-1}_+ \to C(\mathbb{T}^n_\theta)$ defined by
\begin{equation*} f_i(t_1, \ldots, t_n) = t_i u_i^M \end{equation*}
meet the boundary conditions \Cref{eq:polar} of $\thetasphere{2n-1}{\theta}$ and represent odd, normal, commuting elements with $f_1f_1^* + \ldots + f_n f_n^* = 1$. This splits into the sum-square of $2n$  self-adjoint, odd, commuting elements, so $\strivdim{\thetasphere{2n-1}{\theta}}{\mathbb{Z}/2} \leq 2n-1$. 

If $k$ is even, then the final generator is self-adjoint, so a similar argument applies with one fewer term.
\end{proof}

\setstretch{1.06} 

To say more about the strong dimension of $\theta$-spheres, we will discuss quantum tori and the polar decomposition more carefully. Consider a noncommutative 2-torus $A_\theta = C(\bT^2_\theta)$, generated by $u:=u_1$ and $v:=u_2$, for a {\it rational} parameter
\begin{equation*}
  \theta=\frac pq,\quad p\in \bZ,\ q\in \bZ^+
\end{equation*} 
written in lowest terms. From \cite[Theorem 3.12]{rief_canc} (or \cite[Proposition 12.2]{gvf_ncg}) we see that, in this case, there is a rank-$q$ vector bundle $\cE$ on $\bT^2$ such that
\begin{equation}\label{eq:athetae}
  A_{\theta} \,\,
  \cong \,\,
  \End(\cE)
  \,\, = \,\,
  \Gamma(\cE\otimes \cE^*)
  \,\, := \,\,
  \text{continuous sections of the bundle }\cE\otimes \cE^*.
\end{equation}
This is also recalled in \cite[\S 2.5]{khal_lecncg}, which is perhaps more accessible. The center
\begin{equation*}
  Z:=Z(A_{\theta}) \cong C(\bT^2)
\end{equation*}
is generated by the two unitaries $u^q$ and $v^q$ \cite[Corollary 12.3]{gvf_ncg}, and $A_{\theta}$ is a bundle of $q\times q$ matrix algebras over that center. This also makes $A_{\theta}$ into a {\it rank-$q^2$ Azumaya algebra} (see \cite[\S\S 13.7.6 and 13.7.13]{mr}, \cite[\S\S III.1.3 and III.1.4]{bg_algqg}, \cite[Definition 5.4.17]{agpr_pi}, etc.): it is finitely-generated projective over its center $Z$, with
\begin{equation*}
  A_{\theta}\otimes_Z A_{\theta}^{op}\ni a\otimes b \, \xmapsto[\quad\cong\quad]{} \, a\cdot b \in \End_{Z}(A_{\theta}).
\end{equation*}
We next clarify a claim made in the literature about rational $A_\theta$.

\begin{remarks}\label{res:chern}
  \begin{enumerate}[(1),wide=0pt]

  \item\label{item:notflat} The discussion following \cite[Exercise 2.7]{khal_lecncg} claims that the rank-$q$ vector bundle $\cE$ on $\bT^2$ yielding \Cref{eq:athetae} is {\it flat}, which it cannot be. A flat bundle (see \cite[\S 1.2]{kob_cplx}) admits a connection with vanishing curvature, or local trivializations with locally constant transition functions, and flat bundles have vanishing Chern classes by \cite[Proposition II.3.1]{kob_cplx}. In contrast, the Chern class
    \begin{equation}\label{eq:c1}
      c_1(\cE)\in H^2(\bT^2,\bZ)\cong \bZ
    \end{equation}
    is identified in \cite[Proposition 3.8 and Theorem 3.9]{rief_canc} with an integer $a$ for which $ap+bq=1$ (the {\it twist} of \cite[p.299]{rief_canc} is $-a$). In particular, that Chern class cannot vanish. Note that the vector bundle $\cE$ is only unique up to the tensor product with a line bundle \cite[Théorème 9]{dd}, so $a$ is only unique modulo $q$. This is also compatible with the fact that $A_{\theta}$ only depends on the coset of $\theta$ modulo $\bZ$.
    
    Additionally, vector bundles over compact connected Riemann surfaces are completely determined by their ranks and degrees (i.e. Chern classes $c_1$) by \cite[Proposition, p.2]{thad_introtop}. Flat bundles on $\bT^2$, for that reason, are topologically trivial. They need not be holomorphically trivial, but `topologically' is what matters in the present context.
    

  \item\label{item:mnisflat} What does hold, and what was perhaps meant by the flatness claim in the paragraph preceding \cite[Exercise 2.8]{khal_lecncg}, is that the \emph{endomorphism} bundle $\cE\otimes \cE^*$ (recovering $A_{\theta}$ as the section algebra \Cref{eq:athetae}) is flat, and in fact is topologically trivial as a vector bundle (but {\it not} as a bundle of algebras!). The first Chern class $c_1(\cE\otimes \cE^*)\in H^2(\bT^2,\bZ)$ vanishes by \cite[(1.14)]{kob_cplx}, whence triviality by \cite[Proposition, p.2]{thad_introtop} again.
  \end{enumerate}  
\end{remarks}

Note that \Cref{eq:athetae} implies that for rational $\theta=\frac pq$ in lowest terms, the algebras $A_{\theta}:=C(\bT^2_{\theta})$ satisfy the {\it polynomial identities} \cite[Definition 2.2.1]{agpr_pi} of $M_q$ (and no smaller $q'<q$), and hence are {\it polynomial-identity (or PI) algebras} of {\it PI-degree $q$}, as in \cite[\S\S I.13.1 and I.13.3]{bg_algqg}. The polar decomposition \Cref{eq:contpart} of the $\theta$-spheres implies that the same claim applies to the rational sphere algebras. From \cite[Example I.13.2 3]{bg_algqg}, this is in fact a consequence of the claim that the rational $\theta$-spheres are finitely-generated modules over their respective centers.


\setstretch{1}

\begin{proposition}\label{pr:s3thetapi}
  Let $\theta=\frac pq$ be a lowest-terms rational number. Then the following claims hold.
  \begin{enumerate}[(1)]

  \item\label{item:fincent} The center of $\thetasphere{3}{\theta}$ is
    \begin{align*}
      Z_{\theta}:=Z\left(\thetasphere{3}{\theta}\right)
      &\cong
        \mathrm{Cont}_{\partial}\left(\bS^1_{+}\xrightarrow{\quad} Z(C(\bT^2_{\theta}))\right)\numberthis\label{eq:cents3theta}\\
      &=
        \left\{f: [0,1]\xrightarrow{\emph{cont.}}C^*(u^q,v^q)\cong C(\bT^2)\quad |\quad f(0)\in C^*(u^q),\ f(1)\in C^*(v^q)\right\}\\
      &\cong
        C(\bS^3).
    \end{align*}

  \item\label{item:ispi} $\thetasphere{3}{\theta}$ is a PI algebra of PI-degree $q$.

  \item\label{item:azloc} If $\theta \not\in \mathbb{Z}$, then the {\it Azumaya locus} \emph{\cite[\S III.1.7]{bg_algqg}} of $\thetasphere{3}{\theta}$ is the complement in $\bS^3$ of two disjoint circles
    \begin{equation*}
      \bS^1_i\subset \bS^3\cong\mathrm{Prim}(Z_{\theta}),\quad i=0,1
    \end{equation*}
    defined by evaluating the functions $f$ of \emph{\Cref{eq:cents3theta}} at $i=0,1$ respectively.

  \item\label{item:awayazloc} The fibers of the $C(\bS^3)$-algebra $\thetasphere{3}{\theta}$  at the points of the circles $\bS^1_i$ are abelian $q$-dimensional. 
    
  \end{enumerate}
\end{proposition}
\begin{proof}

 {\bf \Cref{item:fincent}}: That \Cref{eq:cents3theta} is indeed the center of $\thetasphere{3}{\theta}$ follows from the polar decomposition \Cref{eq:contpart}, realized in this case as
    \begin{equation}\label{eq:cthetaatheta}
      \thetasphere{3}{\theta}
      \cong
      \left\{f: [0,1]\xrightarrow{\textrm{cont.}} C(\bT^2_\theta) \quad |\quad f(0)\in C^*(u),\ f(1)\in C^*(v)\right\},
    \end{equation} 
and the fact that $u^q$ and $v^q$ generate the center of $C(\bT^2_\theta)$.

{\bf \Cref{item:ispi}}: This is immediate from \Cref{eq:cthetaatheta} and the already-noted analogous claim about $C(\bT^2_{\theta})$.

{\bf \Cref{item:azloc} and \Cref{item:awayazloc}}: Evaluating the $f$ of \Cref{eq:cthetaatheta} at any interior point $t\in (0,1)$ produces all of $C(\bT^2_{\theta})$, which in turn has a $q\times q$ matrix algebra as its fiber at any point of
    \begin{equation*}
      \bT^2\cong\mathrm{Prim}(Z(C(\bT^2_\theta)))
    \end{equation*}
    by \Cref{eq:athetae}. On the other hand, evaluating the $f$ of \Cref{eq:cthetaatheta} at either boundary point produces a function algebra $C(\bS^1)$ generated, respectively, by one of the unitaries $u,v\in C(\bT^2_\theta)$. Focusing on $u$ to fix the notation, further evaluation at a point $z$ of the corresponding circle $\bS^1$ means imposing the relation $u^q=z$. The $C^*$-algebra generated by a unitary $u$ subject to the relation $u^q=z\in \bS^1$ is indeed abelian and $q$-dimensional, and we are done.
\end{proof}

\begin{remark}\label{re:heeg}
  The identification given by \Cref{pr:s3thetapi} item \Cref{item:azloc} of the ``ill-behaved'' locus of
  \begin{equation*}
    \bS^3\cong\mathrm{Prim}(Z(\thetasphere{3}{\theta}))
  \end{equation*}
  with two disjoint circles meshes well with various familiar aspects of the geometry/topology of the 3-sphere.
  
  In the usual realization \cite[\S 1]{mats_3sph-1} of $\bS^3$ as a union of two solid tori $\bD^2\times \bS^1$ glued along their common boundary $\bS^1\times \bS^1$ (as familiar from the theory of {\it Heegaard splittings} \cite[\S 2]{hemp_3mfld}), the two exceptional circles of \Cref{pr:s3thetapi} item \Cref{item:azloc} can be thought of as the two ``core'' circles
  \begin{equation*}
    \bS^1\cong \{0\}\times\bS^1\subset \bD^2\times \bS^1
  \end{equation*}
  of the two solid tori. The complement
  \begin{equation*}
    \bS^3\setminus\left(\bS_0^1\sqcup \bS_1^1\right)
  \end{equation*}
  can be thought of as $\bT^2\times (0,1)$, as expected from \Cref{eq:cents3theta} after removing the problematic boundary points of the unit interval.
  
  Alternatively, given the {\it Hopf fibration} (see the discussion preceding \cite[Proposition 17.22]{bt_forms}) of $\bS^3$ over $\bS^2$ with fiber $\bS^1$, the two circles can be thought of as the fibers above two antipodal poles $\pm p$ of $\bS^2$. The complement is then a {\it trivial} $\bS^1$-fibration over
  \begin{equation*}
    \bS^2\setminus\{\pm p\} \,\cong \, \bS^1\times (0,1),
  \end{equation*}
  which is again identifiable with $\bT^2\times(0,1)$.
\end{remark}

The following result shows an opposite extreme to \Cref{th:strongtheta}. 

\begin{proposition}\label{pr:exactlypq}
  For any lowest-terms rational $\theta=\frac pq$, $\strivdim{\thetasphere{3}{\theta}}{\bZ/q} = \infty$.
\end{proposition}
\begin{proof}
  We have from \Cref{eq:athetae} that there is a realization
  \begin{equation*}
    C(\bT^2_{\theta})\cong \Gamma(\cE\otimes \cE^*)
  \end{equation*}
  for a rank-$q$ vector bundle $\cE$ over $\bT^2$. Furthermore, \Cref{pr:s3thetapi} (together with \Cref{re:heeg}) describes the spectrum of the center $Z_{\theta}$ of $\thetasphere{3}{\theta}$ as
  \begin{equation}\label{eq:s3dec}
    \begin{aligned}
      \bS^3&\quad\cong\quad
             \left(\bD^2\times \bS^1\right)
             \cup_{\bT^2}
             \left(\bS^1\times \bD^2\right)\\
           &\quad\cong\quad
             \bS_0^1\ \sqcup\ \bT^2\times(0,1)\ \sqcup\ \bS_1^1,
    \end{aligned}
  \end{equation}
  with $\thetasphere{3}{\theta}$ a $q\times q$-matrix bundle over the Azumaya bulk $\bT^2\times (0,1)$ and a $\bC^q$-bundle over each of the two exceptional circles $\bS_i^1$, $i \in \{0, 1\}$. We can then recover $A_{\theta}$ as the quotient of $\thetasphere{3}{\theta}$ over the ``middle'' slice
  \begin{equation}\label{eq:coreslice}
    \bT^2\cong \bT^2\times\left\{\frac 12\right\}\subset \bT^2\times (0,1)
  \end{equation}
  of \Cref{eq:s3dec}.
 
  Assuming the strong local-triviality dimension is $d<\infty$, consider commuting normal elements $x_0, \ldots, x_d$ in the $\zeta:=e(1/q)$ eigenspace of $\thetasphere{3}{\theta}$ with 
  \begin{equation*}
    \sum_{j=0}^{d} x_jx_j^* \, = \, 1.
  \end{equation*}
  Over each point in the Azumaya locus $\bT^2\times (0,1)$, said elements specialize to commuting $q\times q$ matrices, of degree $\zeta$ with respect to the $\bZ/q$-action on $M_q$, and not all of these are zero. It follows that every such tuple splits the $\bC^q$ on which the corresponding $M_q$ acts into a direct sum of lines, namely, the $t\zeta^k$-eigenspaces of {\it any} of the non-zero $x_j$ for $0\le k\le q-1$ and $0<t\le 1$. Said spaces are the same for different $x_j$ because these normal elements commute. Further, the splitting is determined up to a cyclic permutation, as the $\bZ/q$-action maps the $\zeta^k$-eigenspace of $x_i$ onto its $\zeta^{k+1}$-eigenspace.
  
  The space of {\it ordered} splittings of $\bC^q$ as a direct sum of lines is the {\it full flag manifold} \cite[\S 3]{arv_flag} $U(q)/U(1)^q$, easily proven simply-connected via the long exact homotopy sequence \cite[\S 17.3]{steen_fib}. The quotient by the cyclic $\bZ/q$-action then produces a space $\bP_q$ with fundamental group $\bZ/q$, and hence a $\bP_q$-bundle $\cE_q$ over $\bT^2\times (0,1)$. Moreover, the $x_j$ jointly give, as described above, a single section of $\cE_q$ over all of $\bT^2\times (0,1)$.
  
  Varying the slice
  \begin{equation*}
    \bT^2\times\{t\}\subset \bT^2\times (0,1)
  \end{equation*}
  thus gives a homotopy $(s_t)_{t\in (0,1)}$ of sections of $\cE_q$ over the ``core'' slice $\bT^2\times\left\{\frac 12\right\}$ of \Cref{eq:coreslice}. On slices $\bT^2 \times \{\varepsilon\}$ with small $\varepsilon\in (0,1)$ (so, close to one of the two exceptional circles) $s_{\varepsilon}$ will be uniformly close to the section $s_u$ resulting from the eigenspaces of one of the generators $u$ of $A_{\theta}$. Similarly, $s_{1-\varepsilon}$ will be close to the section $s_v$ of $\cE_q$ induced by the {\it other} generator $v\in A_{\theta}$. So, $s_u$ and $s_v$ are homotopic.

The previous claim results in a contradiction. If one restricts everything in sight to the circle $\bS^1\subset \bT^2$ corresponding to $u$, sections of $\cE_q$ are nothing but loops in the fiber $\bP_q$ over $1\in \bS^1$. As such, $s_u$ is {\it not} nullhomotopic (as it generates $\pi_1(\bP_q)\cong \bZ/q$), but $s_v$ is. 
\end{proof}

\begin{remark}\label{re:whyjustq}
  The reason the argument employed in the proof of \Cref{pr:exactlypq} only goes through for $\theta$ with denominator {\it exactly} (rather than divisible by) $q$ is that for higher-dimensional $\bC^{kq}$, the $x_i$ would not produce an unambiguous splitting into $q$ subspaces. Rather, the (non-)vanishing patterns of the $x_i$ on various subsets of $\bT^2\times (0,1)$ would force us to make choices over which eigenspaces we select, and those choices might not be compatible.
\end{remark}

  \begin{remark}\label{re:alsonotsup}
    The strict inequality $\wtrivdim{A_{\theta}}{\bZ/2}>0$, which is equivalent to \cite[Proposition 1.9]{bentheta}, provides yet another natural family of examples showing that \Cref{eq:supax} is generally not an equality.
    
    Consider noncommutative tori $A_{\theta}$ with lowest-terms rational $\theta=\frac pq$ and even $q$, acted upon by $\bZ/2$ in the usual generator-negating fashion. The matrix-bundle realization \Cref{eq:athetae} then casts $A_{\theta}$ as a $\bZ/2$-$C(\bT^2)$-algebra with
    \begin{equation*}
      C(\bT^2):=C^*(u_1^q,u_2^q)\subset C^*(u_1,u_2) = A_{\theta},
    \end{equation*}
    and $\bZ/2$ is easily seen to act on every fiber $M_q$ as conjugation by a unitary involution with equidimensional $(\pm 1)$-eigenspaces. The fiber-wise LT dimensions thus vanish, so that
    \begin{equation*}
      \wtrivdim{A_{\theta}}{\bZ/2}
      >
      \sup_{x\in \bT^2} \wtrivdim{A_{\theta,x}}{\bZ/2}
      =
      \sup_{x\in \bT^2} \wtrivdim{M_q}{\bZ/2}
      =0
    \end{equation*}
    and similarly for the other LT dimensions. 
  \end{remark}

\section*{Acknowledgments}

The views expressed in this document are those of the authors and do not reflect the official policy or position of the U.S. Naval Academy, the Department of the Navy, the Department of Defense, or the U.S. Government.

\vskip .1 in

\noindent The authors would like to thank the referee for helpful comments.

\vskip .1 in

\noindent A.C. was partially supported by NSF grant DMS-2001128.

\vskip .1 in

\noindent This work is part of the project Graph Algebras partially supported by EU grant HORIZON-MSCA-SE-2021 Project 101086394.



\addcontentsline{toc}{section}{References}

\def\polhk#1{\setbox0=\hbox{#1}{\ooalign{\hidewidth
  \lower1.5ex\hbox{`}\hidewidth\crcr\unhbox0}}}
  \def\polhk#1{\setbox0=\hbox{#1}{\ooalign{\hidewidth
  \lower1.5ex\hbox{`}\hidewidth\crcr\unhbox0}}}
  \def\polhk#1{\setbox0=\hbox{#1}{\ooalign{\hidewidth
  \lower1.5ex\hbox{`}\hidewidth\crcr\unhbox0}}}
  \def\polhk#1{\setbox0=\hbox{#1}{\ooalign{\hidewidth
  \lower1.5ex\hbox{`}\hidewidth\crcr\unhbox0}}}
  \def\polhk#1{\setbox0=\hbox{#1}{\ooalign{\hidewidth
  \lower1.5ex\hbox{`}\hidewidth\crcr\unhbox0}}}


\Addresses

\end{document}